\documentclass[reqno, 11pt]{amsart}  
 
\usepackage{bm}
 \newcommand{\mysection}[1]{\section{#1}
 \setcounter{equation}{0}}

\newtheorem{theorem}{Theorem}
\newtheorem{lemma}[theorem]{Lemma}
\newtheorem{corollary}[theorem]{Corollary}

\theoremstyle{definition}

\theoremstyle{remark}

\newtheorem{remark}[theorem]{Remark}
 
\newcommand\loc{\textnormal{loc}}

\newcommand{\vsharp}{\asymp\kern -.5em\|}

 \makeatletter
 \def\dashintindex{\operatorname%
 {-\kern-.7em\DOTSI\intop\ilimits@}}%
 \makeatother

\newcommand\sft{{\sf t}}

\newcommand\bR{\mathbb{R}}

\newcommand\cL{\mathcal{L}}

  {{}}

\renewcommand\){{\rm)}}

\def\+){\tmspace+\thinmuskip{.05em}\)}

\def\dashnorm{\,\,\text{\bf--}\kern-.5em\|}

\newcommand{\nliminf}{\operatornamewithlimits{\underline{lim}}}

\newcommand{\sign}{\text{\rm\,sign}\,}

\renewcommand{\eqref}[1]{\text{(\ref{#1})}}

\begin{document}

\title[Heat equation with singular drift]{Heat equation with singular drift
which has bounded solution discontinuous
at an interior point}
\author[]{N.V. Krylov}
\address{School of Mathematics, University of Minnesota, Minneapolis, MN, 55455}
\email{nkrylov@umn.edu}



\begin{abstract}
 The goal of this paper is to present a one dimensional heat equation
with a singular drift for which solutions
of the Cauchy problem with smooth initial data
are bounded and are discontinuous at an interior point (common to all solutions).

\end{abstract}

\maketitle

\mysection{Introduction}

This paper is a contribution to a very active research field
of constructing solutions of It\^o equations with singular
drift term and investigating the elliptic
and parabolic equation associated with them.
 We sent the reader to the bibliography
at the end of the paper and to numerous references
in each item. Our drift term does not satisfy
the requirements in those papers apart from \cite{Kr_27}
and the corresponding solutions of It\^o equation
is not weakly unique, which entails that the corresponding
Markov process is not a Feller process and
the solutions of the corresponding differential
equations in domains develop bounded discontinuities even
inside the domain.

We present some
pathological properties of the operator
\begin{equation}
                           \label{6.10.1}
\cL u(t,x)=\partial_{t}u(t,x)+(1/2)D^{2}u(t,x)+b(t,x)Du(t,x) 
\end{equation}
 which seem quite unusual from what 
is expected from a common parabolic operator. Here  $(t,x)\in\bR^{2}$, $\partial_{t}
=\partial/\partial t$, $D
=\partial/\partial x$,
$$
b(t,x)=I_{t>0}t^{-\nu}\sign x,\quad
\nu\in(1/2,1).
$$

Not long ago M\'at\'e Gerencs\'er drew my attention to the following one-dimensional
stochastic equation
\begin{equation}
                            \label{9.7.1}
x_{t}=w_{t}+\int_{0}^{t}b(s,x_{s})\,ds.
\end{equation}
It turns out that this equation has many solutions with different finite-dimen\-sional
distributions (no weak uniqueness). The reason for that is that, if we take
$x_{0}>0$ and consider
$$
x_{t}=x_{0}+w_{t}+\int_{0}^{t}s^{-\gamma}\sign  x_{s}\,ds,
$$
then before $x_{t}$ reaches the origin
we have $x_{t}=x_{0}+(1-\gamma)^{-1}t^{1-\gamma}$ and, since $|w_{t}|$
 is much less than $(1-\gamma)^{-1}t^{1-\gamma}$ for small $t$,
there is $\varepsilon>0$ such that with probability $\geq 3/4$
for all $t\leq \varepsilon$
$$
w_{t}+(1-\gamma)^{-1}t^{1-\gamma}>(1/2)(1-\gamma)^{-1}t^{1-\gamma}.
$$
Therefore, there is a solution (of the equation with $x_{0}=0$) such that 
$$
P(x_{t}>(1/2)(1-\gamma)^{-1}t^{1-\gamma},t\leq \varepsilon)\geq 3/4.
$$
Similarly, there is a solution (of equation with $x_{0}=0$) such that 
$$
P(x_{t}<-(1/2)(1-\gamma)^{-1}t^{1-\gamma},t\leq \varepsilon)\geq 3/4.
$$
Hence, no weak uniqueness.

On the other hand, since
$$
\rho\Big(\frac{1}{\rho^{2}}\int_{0}^{\rho^{2}}t^{-\nu q}\,dt\Big)^{1/q}\to0
$$
as $\rho\to\infty$ and $q\nu<1$
for a $q>1$
there is a   strong Markov $\bR^{2}$-valued diffusion process $X=((\sft_{s},x_{s}), P_{t,x})$
corresponding to the operator $\cL$ with trajectories issued from the point $(t,x)$
having the distribution $P_{t,x}$ given by the distribution of a  solution  of the system
\begin{equation}
                                 \label{8.7.2}
\sft_{s}=t+s,\quad
x_{s}= x+w_{t}+\int_{0}^{s}(\sft_{r})_{+}^{-\gamma}\sign  x_{r}\,dr
\end{equation}
(see Theorem 1.5 and Remark 1.8 of \cite{Kr_27}). 
Let us call such $X$ an $\cL$-process.
It turns out that $\cL$-processes are not unique,
meaning that their finite-dimensional distributions
are not identical.
Indeed, the way one constructs an $\cL$-process   is based
on the fact that, if there is a solution
of \eqref{9.7.1} and an $f\in C^{\infty}_{0}
(\bR^{2})$, then there exists  $\cL$-processes $X',X''$ such that
$$
E'_{0,0}\int_{0}^{\infty}f(\sft_{s},x_{s})\,ds
\geq E \int_{0}^{\infty}f(s,x_{s})\,ds
\geq E''_{0,0}\int_{0}^{\infty}f(\sft_{s},x_{s})\,ds.
$$
We know that there are two solutions of
\eqref{9.7.1} with different distributions,
 therefore, the middle term in the above inequalities is not uniquely defined
and the distributions of $X'$ and $X''$
are different. As a standard consequence of
the above nonuniqueness we have the following.

\begin{remark}
                                  \label{remark 9.8.1}
For  $p,q\in(1,\infty)$ denote by $L_{(p,q)}$ the space of functions $f(t,x)$
on $(0,1)\times \bR$ such that
$$
\|f\|_{L_{(p,q)}}=\begin{cases}
\Big(\int_{0}^{1}
\Big(\int_{\bR}|f(t,x)|^{p}\,dx\Big)^{q/p}
\,dt\Big)^{1/q}<\infty\quad\text{if}\quad p\geq q;\\
\Big(\int_{\bR}
\Big(\int_{0}^{1}|f(t,x)|^{q}\,dt\Big)^{p/q}
\,dx\Big)^{1/p}<\infty\quad\text{if}\quad p\leq q.
\end{cases}
$$
Call $p,q\in(1,\infty)$ admissible if
$$
\frac{1}{p}+\frac{1}{q}\leq 1.
$$
 Also let  $C^{1,2}_{0} $ be the set of continuous functions $u$
on $[0,1]\times \bR$, whose
partial derivatives
$\partial_{t}u,D^{2}u,Du$ are continuous and bounded in 
 $[0,1]\times \bR$,  $u(1,\cdot)=0$, and $u(t,x)=0$
 if $|x|$ is large.
 
 Then it turns out that $M:=\{\cL u:u\in C^{1,2}_{0} \}$
 is not dense in   $L_{(p,q)}$ if   $p,q$ are admissible.
 
 Indeed,   by It\^o's formula for an $\cL$-process   $X$ 
 and $u\in C^{1,2}_{0}$
 $$
 I(f):=E_{0,0}\int_{0}^{1}f(t,x_{t})\,dt=u(0,0)
 $$
 given that $\cL u=-f$, implying that the left-hand side
 is the same for any $\cL$-process. Since we know from Theorem 2.10 of \cite{Kr_27}
 that he functional $I(f)$ is bounded on  
 each $L_{(p,q)}$ if   $p,q$ are admissible, the denseness
 of $M$ would imply that $I(f)$ is the same for all
 $X$ if $f\in L_{(p,q)}$, which is false.

\end{remark}

The author set himself a problem to find how the above nonuniqueness
reflects in the properties of $\cL$. For instance,
a ``reasonably unique'' solution of the Cauchy problem
\begin{equation}
                           \label{5.30.1}
\cL u(t,x)=0
\end{equation}
in $ (-\infty,1)\times \bR$
with terminal condition
\begin{equation}
                           \label{5.30.2}
u(1,x)=g(x),
\end{equation}
 should be written as 
\begin{equation}
                              \label{9.7.2}
u(t,x)=E_{t,x}g(x_{1-t}),
\end{equation}
where $x_{s}$ solves \eqref{8.7.2}.

But what if $t=x=0$? Which solution out of many
should we take in this representation?
We will see that \eqref{9.7.2} holds
at all points in $ (-\infty,1)\times \bR$
apart from the origin at which point
the ``reasonably unique'' $u$ is discontinuous.
It turns out that the unique
solutions of the Cauchy problem \eqref{5.30.1}-\eqref{5.30.2} can develop bounded discontinuities
at the origin even when it is inside the domain $ (-\infty,1)\times \bR$.

 To show that we take $g(x)=\zeta(|x|)\sign x$, where
  $\zeta(t)\geq0$ is infinitely differentiable, nondecreasing,
is zero for $t\leq1/2$, equals one for $t\geq1$.
Set
$$
B(t)=\int_{-\infty}^{t}b(s)\,ds.
$$

\begin{theorem}
                      \label{theorem 5.30.2}
(i) There is a unique bounded 
and continuous function $u$ defined in
$G:=\big((-\infty,1]\times\bR\big)
\setminus(0,0)$, odd with respect to $x$, 
having there Sobolev derivatives
$\partial_{t}u,Du,D^{2}u$, satisfying
\eqref{5.30.1}
in $ (-\infty,1)\times \bR$ (a.e) and
\eqref{5.30.2} on $\bR$, and possessing the following properties:

(a) the function $Du$ is   continuous 
in $G$, all other derivatives of $u$ with respect to $x$ are continuous in $\big((-\infty,1)\times\bR\big)\setminus
\big([0,1]\times\{0\}\big)$;

(b) the function $u-g$ is of class $W^{1,2}_{p}
\big((\delta,1)\times \bR \big)$ for any $p>1$ and $\delta\in(0,1)$, and in $ (-\infty,
-\delta]
\times \bR $ the function $u$ is infinitely differentiable in $(t,x)$ with each derivative bounded for any $\delta>0$;

(c) the function $u(t,x)$ is an increasing
function of $x$ for any $t\leq 1$;

(d) the function $v(t,x):=u(t,x+B(t))$ is 
bounded and continuous in $G$, infinitely differentiable in $(t,x)$ in
$\Gamma:=\big((-\infty,1)\times\bR\big)\cap\{(t,x):x+B(t)>0\}$ and satisfies
there
\begin{equation}
                              \label{5.30.3}
\partial_{t}v+(1/2)D^{2}v=0;
\end{equation}

(ii) the limit of $u(t,x)$ as $(t,x)\to(0,0)$
does not exist.
\end{theorem}

{\bf Proof of  assertion (i)
in Theorem \ref{theorem 5.30.2}}.
{\em Uniqueness\/}. We have to prove that if $g=0$,
then $u=0$. Observe that in equation \eqref{5.30.1} the coefficient $b$ is bounded in any
$[\varepsilon,1]\times \bR$, $\varepsilon\in(0,1)$. Therefore, any bounded continuous solution of class
$W^{1,2}_{p,\loc}([\varepsilon,1]\times \bR)$
vanishing at $t=1$ is zero. By continuity
 $u(0,x)=0$ for $x\ne0$.   Bounded solutions
of the heat equation vanishing on the boundary
are zeros and this proves uniqueness.

{\em Existence\/}. Since the lateral boundary
of $\Gamma$ has a horizontal tangent line
at $(0,0)$, we need an auxiliary construction.
For small $\delta\in(0,1)$ define 
$(t_{\delta},0)$ ($t_{\delta}<0$) as the point of the
intersection of the $t$-axes with the tangent line to $(t,B(t))$ at $(\delta,B(\delta))$
and let
$B_{\delta}(t)=B(t)$ for $t\not\in[t_{\delta},\delta]$
and for $t\in[t_{\delta},\delta]$ let the graph of $B_{\delta}(t)$ be 
the straight segment between the point
 $(\delta,B(\delta))$ and
the point $(t_{\delta},0)$.

Since  $b_{\delta}$ is bounded   and $Dg$ has compact support, there exists a (unique) bounded function $u_{\delta}$ such that
$u_{\delta}-g\in W^{1,2}_{p}((-\infty,1]\times \bR)$, and $u_{\delta}$ solves the problem \eqref{5.30.1}-\eqref{5.30.2} with $b_{\delta}$
in place of $t_{+}^{-\nu}$. The uniqueness
implies that $u_{\delta}$ is odd with respect to $x$
and the arbitrariness of $p$   implies that $u_{\delta}$
and $Du_{\delta}$ are bounded and continues in 
 $(-\infty,1]\times \bR $. 
Obviously, for $0<\delta\leq\delta'$
all $u_{\delta}$ coincide on $[\delta',1]\times\bR$. Therefore,
\begin{equation}
                               \label{6.3.1}
u=\lim_{\delta\downarrow0}u_{\delta}
\end{equation}
exists in $(0,1)\times\bR$ and
$u-g$ is of class $W^{1,2}_{p}
\big( (\delta,1)\times \bR\big)  $ for any $p>1$ and $\delta\in(0,1)$.  In particular,
$u,Du$ are continuous in $(0,1]\times\bR$.

Also, by the maximum principle,
$$
|u_{\delta}|\leq 1.
$$

Since $u_{\delta}$ is odd, $u_{\delta}(t,0)=0$ and $u_{\delta}$ solves the modified \eqref{5.30.1}-\eqref{5.30.2} in  $(-\infty,1]\times [0,\infty)$ with nonnegative boundary data.
By the maximum principle $u_{\delta}\geq 0$ in
$(-\infty,1]\times [0,\infty)$, and then
$Du_{\delta}(t,0)\geq 0$.

Next, let $v_{\delta}(t,x)=u_{\delta}(t,x+B_{\delta}(t))$. Then 
$v_{\delta}\geq 0$ satisfies \eqref{5.30.3}
in 
$$
\Gamma_{\delta}:=\big((-\infty,1)\times\bR\big)
\cap\{(t,x):x+B_{\delta}(t)>0\}
$$
 and vanishes on its lateral boundary.
If $0<\delta'<\delta$, then $B_{\delta'}\leq
B_{\delta}$ and $\Gamma_{\delta'}\subset
\Gamma_{\delta}$. Also, $v_{\delta}$
satisfies \eqref{5.30.3} in $\Gamma_{\delta'}$
and is nonnegative on its lateral boundary,
which implies that $v_{\delta'}\leq v_{\delta}$
in $\Gamma_{\delta'}$ and in $\bar \Gamma_{\delta'}$. It follows that
\begin{equation}
                            \label{6.2.1}
v(t,x):=\lim_{\delta\downarrow 0}v_{\delta}(t,x)
\end{equation}
exists for $(t,x)\in \bar \Gamma$.
Obviously, $v=0$ on the lateral boundary
of $\Gamma$ with the exception, perhaps, of the point $(0,0)$,
which does not belong to the union
of all lateral boundaries of $\Gamma_{\delta}$. As the limit of solutions of the heat equation satisfies the same equation, $v$
satisfies \eqref{5.30.3} in $\Gamma$.

The function $v_{\delta}$ is smooth in $\Gamma_{\delta}$ and the bounded continuous function 
$Dv_{\delta}$ 
in $\bar \Gamma_{\delta}$ also is smooth
in $\Gamma_{\delta}$, satisfies \eqref{5.30.3}
and is nonnegative
on the boundary of $\Gamma_{\delta}$. By
the maximum principle $Dv_{\delta}\geq0$
in $\bar \Gamma_{\delta}$ implying that
\begin{equation}
                             \label{6.3.3}
u(t,x):=v(t,|x|-B(t))\sign x
\end{equation}
is an increasing function of $x $. 
Of course, this definition agrees with
\eqref{6.3.1} because
$$
v(t,|x|-B(t))\sign x=\lim_{\delta\downarrow 0}
u_{\delta}\big(t,|x|+B_{\delta}(t)-B(t)\big)\sign x
$$
and if $t\ne0$
$$
v(t,|x|-B(t))\sign x=\lim_{\delta\downarrow 0}
u_{\delta}(t,x ).
$$

Note that $u(t,x)=v(t,|x|)$ for $t\leq0$,
so that on this set $u$ satisfies $\partial_{t}u+D^{2}u=0$. Therefore, in $\big((-\infty,
-\delta]
\times \bR\big)\setminus\{(0,0)\}$ the function $u$ is infinitely differentiable in $(t,x)$ with each derivative bounded for any $\delta>0$

 By the properties of bounded solutions of the heat equation, on any subset
 of $\Gamma $ lying at a strictly positive distance of its   boundary
any derivative of $v $ is bounded.
It follows from \eqref{6.2.1} that on
on any such subset of $\Gamma$
any derivative of $v $ is continuous.
In particular, all $D^{k}u$ are continuous
in $(-\infty,1)\times[\delta,\infty)$ and
in $(-\infty,1)\times(-\infty,-\delta]$.
They are also continuous in $\{t<0\}$ as the
derivatives of solutions of the heat equation.
Thus, they are continuous in $\big((-\infty,1)\times\bR\big)\setminus  
\big([0,1]\times\{0\}\big)$.
The function $Du$ is continuous in $G$
because it is continuous not only in $\big((-\infty,1)\times\bR\big)\setminus
\big([0,1]\times\{0\}\big)$, but in $(0,1]\times\bR$ as well. 

Equation \eqref{6.3.3} implies that for
all $t\leq 1$ apart from $t=0$
$$
\partial_{t}u(t,x)=\big(\partial_{t}v+b(t)Dv\big)\big(t,|x|-B(t)\big) \sign x,
$$
where the right-hand side is locally summable
in $\{t<1,|x|>0\}$ ($\partial_{t}v,Dv$
are continuous) and $u$ is continuous in $G$.
It follows that   the Sobolev
derivative of $u$ with respect to $t$ exists
on $\{t<1,|x|>0\}$ and is given by the above formula. That it exists in $(0,1)\times \bR$
follows from what was said at the beginning of the  proof. Similarly, the Sobolev derivatives
$Du,D^{2}u$ exist in $\{t<1,|x|>0\}\cup\big((0,1)\times \bR\big)$. That the Sobolev derivatives $\partial_{t}u,Du,D^{2}u$ exist
in $\{t<0\}$ follows from the fact that there $u$ is a solution of \eqref{5.30.3}. Finally,
the fact that $u$ satisfies \eqref{5.30.1}
follows from \eqref{5.30.3} and \eqref{6.3.3}.
This completes the proof of assertion (i)
of Theorem \ref{theorem 5.30.2}. \qed

To prove    assertion (ii) of Theorem 
\ref{theorem 5.30.2} we are going to prove that
$$
\nliminf_{n\to\infty}u(2^{-n},2^{-n/2})>0,\quad \big(u(t,0)=0, t\in(0,1]\big).
$$
Observe that for $t\in(0,1]$ and $x>0$
$$
\partial_{t}u(t,x)+(1/2)D^{2}u(t,x)+t^{-\gamma}Du(t,x)=0,\quad u(t,0)=0,\quad
u(1,x)=g(x).
$$
 Approximate $t^{-\gamma}$ by piece-wise constant functions.
For $n=0,1,2,...$
define
$t_{n}=2^{-n}, b_{n}=t_{n }^{-\gamma},  x_{n}=2^{-n/2}$.
The drift $b_{n}$ is set to be applied on
$[t_{n+1},t_{n}) $. Then set
\[
p (t,x )=(2\pi t)^{-1/2}e^{-x^{2}/(2t)},
\]
\[
q_{n}(t,x,y)=[p (t,y-x)-p (t,y+x)]
e^{  b_{n}(y-x)-(1/2)  b_{n}^{2}t}.
\]

It is easy to check the well-known fact that
$q_{n}(t,x,s,y):=q_{n}(s-t,x,y)I_{s>t}$ is the fundamental solution
of   for $\partial_{t}v+(1/2)D^{2}v+  b_{n} Dv=0$ in
$  (0,\infty)^{2}$ with zero condition on the $x$ axis.

\begin{lemma}
                        \label{lemma 5.31.1}

For $n=0,1,2,...$, $t\in[t_{n+1},t_{n}) $ and $x\geq0$
\begin{equation}
                         \label{6.13.3}
u(t,x)\geq v_{n}(t,x):= \int_{0}^{\infty}u(t_{n},y)q_{n}(t_{n} -t,x,y)\,dy.
\end{equation}
\end{lemma}

Proof. We have $u-v_{n}=0$ on $[t_{n+1},t_{n}]\times\{0\}$ and on $\{t_{n}\}\times [0,\infty)$ and inside $(t_{n+1},t_{n})
\times(0,\infty)$ it holds that
$$
0=\partial_{t}(u-v_{n})+(1/2)D^{2}(u-v_{n})+ b_{n }D(u-v_{n})+(b-b_{n })Du
$$
$$
\geq \partial_{t}(u-v_{n})+(1/2)D^{2}(u-v_{n})+ b_{n }D(u-v_{n}),
$$
where the inequality is due to $b\geq b_{n }$ and $Du\geq0$. Then \eqref{6.13.3} follows from the maximum
principle. \qed

Set $u_{n}=u(t_{n},x_{n})=u(2^{-n},2^{-n/2})$,
$$
\alpha_{n}= \int_{x_{n}}^{\infty} q_{n}(t_{n+1},x_{n+1},y)\,dy
$$
$$
=t_{n+1}^{1/2}\int
_{-z_{n}}^{\infty}
q_{n}(t_{n+1},x_{n+1},zt_{n+1}^{1/2}+x_{n+1}+  b_{n}t_{n+1})\,dz,
$$
where  $z_{n}:= b_{n} t_{n+1}^{1/2}+1-\sqrt 2$. Since, obviously, $q_{n}(t,x,y)>0$ for $x,y>0$, we have $\alpha_{n}>0$.
\begin{corollary}
                   \label{corollary 5.31.1}
For $n=0,1,2,...$  and $x\geq0$
$$
u(t_{n+1},x)\geq \int_{0}^{\infty}u(t_{n},y)q_{n}(t_{n+1},x,y)\,dy.
$$
In particular, $u_{n+1}\geq\alpha_{n}u_{n}$.
\end{corollary}

{\bf Proof of assertion (ii) in Theorem 
\ref{theorem 5.30.2}}. Since $u(t,0)=0\to
0$ as $t\downarrow 0$, it suffices to prove
that $\nliminf_{n\to\infty}u_{n}>0$. By Corollary \ref{corollary 5.31.1}
$$
\nliminf_{n\to\infty}u_{n}\geq u_{0}
\prod_{n=0}^{\infty}\alpha_{n},
$$
so that it suffices to prove that
\begin{equation}
                         \label{5.4.1}
\prod_{n=0}^{\infty}\alpha_{n}>0.
\end{equation}

We have
$$
\alpha_{n}=\beta_{n}-\gamma_{n},
$$
where  
$$
\gamma_{n}=(2\pi)^{-1/2}\int_{-z_{n}}^{\infty}e^{-(z+2)^{2}/2-2b_{n}x_{n+1}}\,dz
$$
$$
\leq(2\pi)^{-1/2}\int_{-\infty}^{\infty}e^{-(z+2)^{2}/2-2b_{n}x_{n+1}}\,dz
= e^{-2b_{n}x_{n+1}},
$$
$$
\beta_{n}=(2\pi)^{-1/2}\int_{-z_{n}}^{\infty}e^{-z^{2}/2}\,dz=1-(2\pi)^{-1/2}\int_{ z_{n}}^{\infty}e^{-z^{2}/2}\,dz.
$$
As is easy to see, $z_{n}>2-\sqrt2$ and
$(2\pi)^{-1/2}z_{n}^{-1}\leq 1/2<1$. Therefore,
by using the inequality
$$
\int_{x}^{\infty}e^{-z^{2}/2}\,dz\leq x^{-1}e^{-x^{2}/2},\quad x>0,
$$
we get that 
$$
\beta_{n}\geq 1-e^{-z_{n}^{2}/2}.
$$
It follows that for some $n_{0}\geq0$ and all
$n\geq n_{0}$
$$
\alpha_{n}\geq 1-e^{-z_{n}^{2}/2}-e^{-2b_{n}x_{n+1}}\geq 1/2.
$$
Since, as easy to see,
$$
\sum_{n=0}^{\infty}\big(e^{-z_{n}^{2}/2}+e^{-2b_{n}x_{n+1}}\big)<\infty,
$$
we have
$$
\prod_{n=n_{0}}^{\infty}\alpha_{n}>0,
$$
which leads to \eqref{5.4.1} and proves the theorem. \qed

We finish the  article
 with justifying \eqref{9.7.2}.

\begin{theorem}
                       \label{theorem 9.7.2}

For $u$  from Theorem \ref{theorem 5.30.2}
the representation
  \eqref{9.7.2} holds
at all points in $ (-\infty,1)\times \bR$
apart from the origin (at which point
  $u$ is discontinuous).
\end{theorem}

Proof. If $t\in(0,1)$ the result follows
by It\^o's formula because $b(s,x)$
is bounded for $s\in[t,1]$. If $t=0$
and $x\ne0$, then relying on the continuity
of $u$ at $(0,x)$ and the Markov property,
we get
$$
u(0,x)=\lim_{\varepsilon\downarrow0}E_{0,x}
u(\varepsilon,x_{\varepsilon})=
\lim_{\varepsilon\downarrow0}E_{0,x}
E_{\varepsilon,x_{\varepsilon}}g(x_{1-\varepsilon})=E_{0,x }g(x_{ 1}).
$$
Finally, if $t<0$, by definition
and the Markov property we obtain
$$
u(t,x)=E_{t,x}u(0,x_{-t})=
E_{t,x}u(0,x_{-t})I_{x_{-t}\ne0}
=E_{t,x}I_{x_{-t}\ne0}E_{ 0,x_{-t}}g(x_{1}) 
$$
$$
=E_{t,x}I_{x_{-t}\ne0} g(x_{1-t})
=E_{t,x}  g(x_{1-t}).
$$   
The theorem is proved.  \qed

Since the left-hand sife of \eqref{9.7.2}
is discontinuous, we have the following.

\begin{corollary}
                             \label{corollary 9.8.2}
The $\cL$-processes are not Feller processes.
\end{corollary}

 {\bf Author contributions} The paper is written in its entirety by the author
 \medskip
 
{\bf Data Availability Statement} No datasets were generated or analyzed during
the current study.\medskip

{\bf Declarations}\medskip

{\bf Competing interests} The authors declare no competing interests.

\end{document}